\documentclass{amsart}
\usepackage{amssymb}  % for \leqslant, \geqslant
\usepackage{amsmath} % for \pmod with smaller space before (mod ...)
\usepackage{amsthm}  % for \begin{proof} \end{proof}   !!! defines \openbox !!!
\usepackage{bm} % for \bm
\usepackage{stmaryrd} % for \trianglelefteqslant
\usepackage{eucal}     % for a better script font

\advance\hoffset by -1in
\advance\voffset by -1in

\newtheorem{thm}{Theorem}
\newtheorem{lem}[thm]{Lemma}
\newtheorem{prop}[thm]{Proposition}

\newtheorem{prob}{Problem}

\newcommand{\be}{\begin{equation}}
\newcommand{\ee}{\end{equation}}

\newcommand{\la} {\langle}
\newcommand{\ra} {\rangle}

\newcommand{\vf}{\varphi}

\newcommand{\s}{\sigma}
\newcommand{\bt}{\boxtimes}

\renewcommand{\r}{\rho}
\renewcommand{\le}{\leqslant}

\renewcommand{\t}{\tau}

\newcommand{\g}{\gamma}

\newcommand{\rd}{\rightthreetimes}

\newcommand{\nor}{\trianglelefteqslant}

\newcommand{\wt}{\widetilde}

\newcommand{\E}{\mathcal{E}}

\newcommand{\C}{\mathrm{C}}

\newcommand{\M}{\mathcal{M}}

\title{Abelian-by-cyclic Moufang loops}
\author{Alexander N. Grishkov and Andrei V. Zavarnitsine}
\thanks{Supported by FAPESP, Brazil (proc. 2010/51793-2);
by the Russian Foundation for Basic Research
(projects 10--01--90007, 11--01--00456);
by the Council of the President grants (project NSc--3669.2010.1);
by the Program ``Development of the Scientific Potential of Higher School''
(project 2.1.1.10726); by the Russian Federal Program ``Scientific and pedagogic people of the innovative Russia'' (contract 14.740.11.0346).}
\date{}

\begin{document}
\begin{abstract} We use groups with triality to construct a series of nonassociative Moufang loops.
Certain members of this series contain an abelian normal subloop with the corresponding quotient
being a cyclic group. In particular, we give a new series of examples of
finite abelian-by-cyclic Moufang loops. The previously known \cite{raj}
loops of this type of odd order $3q^3$, with prime $q\equiv 1\pmod 3$, are particular cases of our series. Some of the examples are shown to
be embeddable into a Cayley algebra.



% 512.54 % UDC
\end{abstract}
\maketitle

\section{Introduction}

Universal constructions for new Moufang loops are few. An example is Chein's
doubling process \cite{che} which allows one, given an arbitrary nonabelian group $G$,
to obtain a nonassociative Moufang loop of cardinality $2|G|$. Since the discovery
of the relation between groups with triality and Moufang loops, which has been used successfully by
various authors \cite{lie,gz,gag} to solve important problems in the theory Moufang loops,
there appeared new ways of constructing Moufang loops using groups with triality.

In the present paper, we build a new series of groups with triality and then derive an explicit
multiplication formula for the corresponding Moufang loops. In particular, we obtain
a series of {\em abelian-by-cyclic} Moufang loops (i.\,e. an  upward extension of an abelian group by a cyclic
group). To be more precise, let $R$ be an arbitrary associative commutative unital ring and
let $R_0$ be a cyclic group of
invertible elements of $R$. We show that the set of tuples $(r,x,y,z)$, where $r\in R_0$, $x,y,r\in R$,
with the multiplication
\be\label{abm}
\begin{array}{r@{}l}
(r_1,x_1,& y_1,z_1)(r_2,x_2,y_2,z_2)=\\[5pt]
(r_1&r_2,x_1+r_1x_2,y_1+r_1y_2,r_2z_1+z_2+(1-2r_1^{-1}r_2)x_1y_2-x_2y_1)
\end{array}
\ee
is an abelian-by-cyclic nonassociative Moufang loop of the form $R_0.(R+R+R)$ provided that
either $R_0$ has order $3$ or $R$ has characteristic $2$.

The minimal finite loops of this type clearly arise if $R_0$ has prime order $p$ and $R$
is a minimal finite field with an element of multiplicative order $p$. For example,
this gives abelian-by-cyclic proper Moufang
loops of orders $3.2^6$, $7.2^9$, $3.5^6$, $3.7^3$, etc.

The abelian-by-cyclic Moufang loops are of interest in light of the following problem proposed by
M.\,Kinyon and based on \cite{chr}:
\begin{prob} Let $M$ be a Moufang loop with a normal abelian subgroup (i.\,e. associative subloop)
$N$ of odd order such that $M/N$ is a cyclic group of order bigger than $3$.
\begin{enumerate}
\renewcommand{\theenumi}{\roman{enumi}}
\item Is $M$ a group?
\item If the orders of $N$ and $M/N$ are coprime, is $M$ a group?
\end{enumerate}
\end{prob}
Although the finite loops of the form (\ref{abm}) are not counterexamples to this problem, there are reasons
to believe that they are essentially the only types of abelian-by-cyclic Moufang loops such that the orders
of $N$ and $M/N$ are coprime.

Examples of abelian-by-cyclic Moufang loops of odd order $3q^3$ with prime $q\equiv 1\pmod 3$
have also appeared
in \cite{raj}, where the problem of the existence of nonassociative Moufang loops
of orders $pq^3$, with $p,q$ prime, was
considered. Due to the uniqueness result of \cite{raj}, these examples must be particular cases of
our series (\ref{abm}). However, we have not attempted to find an explicit isomorphism.

The loops (\ref{abm}) are constructed as particular cases of a wider class of Moufang
loops $M_{a,b}$, $a,b\in R$, not all of which are abelian-by-cyclic but all have the general
structure $R_0.(R+R).R$ for a given subgroup $R_0\le R^\times$, see Lemma \ref{mult}.
We show that some members of this series can be embedded in the Cayley algebra $\mathbb{O}(R)$.
In the last section, we raise the isomorphy problem for the loops $M_{a,b}$ and prove one relevant result.

\section{Preliminaries}

A loop $M$ is called a {\it Moufang loop} if $xy\cdot zx = (x\cdot yz)x$ for all $x,y,z \in M$.
For basic properties of Moufang loops, see \cite{bru}.

A group $G$ possessing automorphisms $\rho$ and $\sigma$ that satisfy $\rho^3=\sigma^2=(\rho\sigma)^2=1$ is
called a {\it group with triality $\langle\rho,\sigma\rangle$} if
$$
(x^{-1}x^{\sigma})(x^{-1}x^{\sigma})^\rho(x^{-1}x^{\sigma})^{\rho^2}=1
$$
for every $x$ in $G$. In a group $G$ with triality $S=\langle\rho,\sigma\rangle$,
the set $\M(G)=\{ x^{-1}x^{\sigma}\ |\ x\in G\}$
is a Moufang loop with respect to the multiplication
\begin{equation} \label{loop_mult}
m.n=m^{-\rho} n  m^{-\rho^2}
\end{equation}
for all $n,m\in \M(G)$. Conversely, every Moufang loops arises so from a suitable group with triality.
For more information on the relation between groups with triality and Moufang loops, see \cite{gz_tri}.

Every group with triality $G$ possesses a (necessarily unique) maximal normal subgroup contained in ${\rm C}_{G}(S)$
which we will denote by ${\rm Z}_{S}(G)$. For every Moufang loop $M$ there exists a unique group with triality $\E(M)$
that satisfies both ${\rm Z}_S(\E(M))=1$ and $[\E(M),S]=\E(M)$.

A homomorphism $\varphi: G_1\to G_2$ of groups $G_1$ and $G_2$ with triality $S$ is called an {\it
$S$-homo\-morphism} if $\alpha\varphi=\varphi\alpha$ for all $\alpha \in S$. The following result is a consequence of
\cite[p. 383--384]{dor}.
% To be more precise: Discussion on p. 383 before the construction
% of the group $G(S)$, and Corollary 1 on p.384

\begin{lem} \label{mmini} Moufang loops $M_1$ and $M_2$ are isomorphic if and only if $\E(M_1)$ and $\E(M_2)$ are $S$-isomorphic.
\end{lem}

\section{Triality representations}

\begin{lem} \label{lm_id} Let $M$ be a Moufang loop. Then, for every $x,y,m\in M$, we have
$$m^{-1}(mx.y)=xm^{-1}.my=(x.ym^{-1}).m$$
\end{lem}
\begin{proof} This follows
from the left and right Moufang identities.
\end{proof}

\begin{lem} \label{sm_com} Let $G$ be a group with triality $S=\la \s,\r\ra$. Then, for every
$m\in M=\M(G)$, $G$ is a group with triality $S_{(m)}=\la\s,\r^2 m \r^2\ra$ which we denote by $G_{(m)}$.
The Moufang loop $M_{(m)}=\M(G_{(m)})$ has multiplication
\be\label{is_mul}
x*_{(m)}y =(x.m^{-1}).(m.y)
\ee
for all $x,y\in M_{(m)}$. In particular, $M_{(m)}$ is isotopic to $M$ and, conversely, every loop-isotope
of $M$ has the form $M_{(m)}$ for some $m\in M$.
\end{lem}
\begin{proof} By the triality identity we have
\begin{gather*}
(\r^2 m \r^2)^3=\r^2 m \r m \r m \r^2=\r m^{\r^2} m^\r m \r^2=\r^3=1,\\
(\s\r^2 m \r^2)^2=\r m^{-1} \r\r^2 m \r^2=1.
\end{gather*}
Hence, $S_{(m)}$ is an $S_3$-complement for $G$ in $SG$. Note that $M_{(m)}$ coincides with $M=[G,\s]$
as a set. For every $n\in M_{(m)}$, we have

\begin{align*}
nn^{\r^2m\r^2}&n^{(\r^2m\r^2)^2}=n(m^{-1}n^{\r^2}m)^{\r^2}n^{\r m^{-1}\r}=\\
&nm^{-\r^2}n^\r m^{\r^2}(mn^\r m^{-1})^\r=nm^{-\r^2}n^\r(m^{\r^2}m^\r)n^{\r^2}m^{-\r}=\\
&nm^{-\r^2}(n^\r m^{-1}n^{\r^2})m^{-\r}=nm^{-\r^2}(n^{-1}.m^{-1})m^{-\r}=\\
&n((n^{-1}.m^{-1}).m)=nn^{-1}=1
\end{align*}

Hence, $G_{(m)}$ is indeed a group with triality $S_{(m)}$. The multiplication formula in $M_{(m)}$ is then given by
\begin{align*}
x*_{(m)}y =& x^{-\r^2m\r^2}yx^{\r m^{-1}\r} = (m^{-1}x^{-\r^2}m)^{\r^2}y(mx^{-\r}m^{-1})^\r=\\
&m^{-\r^2}x^{-\r}(m^{\r^2}ym^\r) x^{-\r^2}m^{-\r}=m^{-\r^2}(x^{-\r}(y.m^{-1}) x^{-\r^2})m^{-\r}=\\
&m^{-\r^2}(x.(y.m^{-1}))m^{-\r}=(x.(y.m^{-1})).m=(x.m^{-1}).(m.y)
\end{align*}
where the last equality holds by Lemma \ref{lm_id}.

By \cite[Lemma VII.5.8]{bru} every loop-isotope of a Moufang loop is isomorphic to a
principal isotope with multiplication of the form (\ref{is_mul}).
\end{proof}

Let $R$ be a commutative ring and  $S=\la \s,\r\ra$. A right $RS$-module $V$ is
called a {\it triality module} ({\it for $S$}) if $V$ is a group with triality $S$. This holds if and only if
$(\s-1)(1+\r+\r^2)$ annihilates $V$. The classification
of the indecomposable triality $RS$-modules over fields is given in \cite[Lemma 5]{syl}.

Let $G$ be a group with triality $S$. A right $R[S\rd G]$-module $V$ is
called a {\it triality module} ({\it for $G$}) if the natural semidirect product $G\rd V$ is a group with triality $S$.

\begin{lem} \label{tr_kr} An $R[S\rd G]$-module $V$ is a triality module for $G$ if and only if
the restriction $V\mid_{S_{(m)}}$ is a triality module for $S_{(m)}$ for every $m\in \M(G)$.
\end{lem}
\begin{proof} If $V$ is a triality module for $G$ then $V\mid_{S_{(m)}}$ is a triality module for $S_{(m)}$
by Lemma \ref{sm_com}, since $\M(G)\le \M(G\rd V)$. Let us prove the converse.

Let $g\in G$ and $v\in V$. Set $x=gv\in G\rd V$. We have
$$[x,\s]=v^{-1}mv^\s=mv^{-m+\s},$$
where $m=[g,\s]$. Since $mm^{\r}m^{\r^2}=1$, we obtain
\begin{align*}
&[x,\s][x,\s]^\r[x,\s]^{\r^2}=mv^{-m+\s}m^\r v^{-m\r+\s\r}m^{\r^2}v^{-m\r^2+\s\r^2}=\\
&mm^\r  m^{\r^2}v^{-mm^\r m^{\r^2}+\s m^\r m^{\r^2}-m\r m^{\r^2}+\s\r m^{\r^2}-m\r^2+\s\r^2}=\\
&v^{-1+\s m^{-1}-\r m^{-1}+\s\r^2 m\r^2-m\r^2+\s\r^2}=v^{(\s-\r)(m^{-1}+\r^2 m\r^2+\r^2)}=\\
&v^{\r^2(\s-1)(\r^2 m^{-1}\r^2+\r m\r+1)\r}=0,
\end{align*}
where the last equality holds
because the operator $(\s-1)(\r^2 m^{-1}\r^2+\r m\r+1)$ annihilates $V$ by the assumption. The claim follows.
\end{proof}

\section{Triality modules and tensor product}\label{sect}

Let $H$ be a group with triality $S$. Denote $\wt{H}=S\rd H$. Let $V_1,V_2,U$ be triality $R\wt H$-modules.
Suppose that $\vf:V_1\otimes V_2\to U$ is an $R\wt H$-module homomorphism. For brevity, we will write $v_1\bt v_2=\vf(v_1\otimes v_2)$, where
$v_1\in V_1$, $v_2\in V_2$. Then, in particular, we have
\be\label{balh}
v_1^h \bt v_2^h = (v_1 \bt v_2)^h
\ee
for all $h\in \wt{H}$. We endow the Cartesian product $W=V_1\times V_2\times U$ with the operation
\be\label{oper}
(v_1,v_2,u)(v_1',v_2',u')=(v_1+v_1',v_2+v_2',u+u'+v_1\bt v_2')
\ee
which turns $W$ into a nilpotent group of class (at most) $2$ with a central subgroup (isomorphic to) $U$. Moreover, setting
$(v_1,v_2,u)^h=(v_1^h,v_2^h,u^h)$
for every $h\in \wt{H}$ and $(v_1,v_2,u)\in W$ (with the $R\wt{H}$-module action of $h$ on the components)
correctly defines an action of $\wt{H}$ on $W$ due to (\ref{balh}).
The resulting group $G=H\rd W$ has a natural $S$-action, and the normal series
$$
1\nor U\nor W \nor G
$$
is $S$-invariant with the corresponding factors being groups with triality~$S$.
In general, the triality on the factors of a normal series of a group does not imply the triality on the whole group.
We obtain the following criterion:
\begin{lem} \label{trnil2}
%Let $H$ be a group with triality $S$ and let $V_1,V_2,U$ be triality $R[S\rd H]$-modules with
%a module homomorphism $\vf:V_1\otimes V_2\to U$. Then, in the above notation,
The group $G$ constructed above is a group with triality $S$ if and only if
$$
l_1^{\r^2m\r^2}\bt l_2^{(\r^2m\r^2)^2}\in \C_U(\s)
$$
for all $m\in \M(H)$, $l_1\in \M(V_1)$, $l_2\in \M(V_2)$.
\end{lem}
\begin{proof} Elements of $G$ will be written as $(h,v_1,v_2,u)$, where $h\in H$, $(v_1,v_2,u)\in W$. Then the
multiplication and inversion in $G$ are given explicitly by
\be\label{pi}
\begin{array}{r@{}l}
(h,v_1,v_2,u)(h',v_1',v_2',u'&)=(hh',v_1^{h'}+v_1',v_2^{h'}+v_2',u^{h'}+u'+v_1^{h'}\bt v_2'), \\[5pt]
(h,v_1,v_2,u)^{-1}=(&h^{-1},-v_1^{h^{-1}},-v_2^{h^{-1}},-u^{h^{-1}}+v_1^{h^{-1}}\bt v_2^{h^{-1}}).
\end{array}
\ee
We now check the triality for $G$. Let $g=(h,v_1,v_2,u)\in G$. Then setting $m=h^{-1}h^\s$ we have
\be\label{gmgs}
g^{-1}g^\s=(m,v_1^{-m+\s},v_2^{-m+\s},u^{-m+\s}+v_1^m\bt v_2^{m-\s}).
\ee
Using the fact that $H$ is a group with triality and that $V_1$, $V_2$, $U$ are triality $R\wt H$-modules we
have
\begin{align*}
(g^{-1}&g^\s)(g^{-1}g^\s)^\r(g^{-1}g^\s)^{\r^2}\\
=&(m,v_1^{-m+\s},v_2^{-m+\s},u^{-m+\s}+v_1^m\bt v_2^{m-\s})\\
\times&(m^\r,v_1^{-m\r+\s\r},v_2^{-m\r+\s\r},u^{-m\r+\s\r}+v_1^{m\r}\bt v_2^{m\r-\s\r})\\
\times&(m^{\r^2},v_1^{-m\r^2+\s\r^2},v_2^{-m\r^2+\s\r^2},u^{-m\r^2+\s\r^2}+
v_1^{m\r^2}\bt v_2^{m\r^2-\s\r^2})\\
=&(mm^\r,v_1^{-mm^\r+\s m^\r-m\r+\s\r},v_2^{-mm^\r+\s m^\r-m\r+\s\r},u^{-mm^\r+s m^\r-m\r+\s\r}\\
&\quad +v_1^{mm^\r}\bt v_2^{mm^\r-\s m^\r}+v_1^{m\r}\bt v_2^{m\r-\s\r}
+v_1^{-mm^\r+\s m^\r}\bt v_2^{-m\r+\s\r})\\
\times&(m^{\r^2},v_1^{-m\r^2+\s\r^2},v_2^{-m\r^2+\s\r^2},
u^{-m\r^2+\s\r^2}+v_1^{m\r^2}\bt v_2^{m\r^2-\s\r^2})=
\end{align*}
[by $mm^\r m^{\r^2}=1$]
\begin{align*}
=(&1,v_1^{-1+\s m^{-1}-\r m^{-1}+\s\r^2 m \r^2 -m\r^2+\s\r^2},
v_2^{-1+\s m^{-1}-\r m^{-1}+\s\r^2 m \r^2 -m\r^2+\s\r^2},\\
& u^{-1+\s m^{-1}-\r m^{-1}+\s\r^2 m \r^2 -m\r^2+\s\r^2}+v_1\bt v_2^{1-\s m^{-1}}
+v_1^{\r m^{-1}}\bt v_2^{\r m^{-1}-\s\r^2 m \r^2}\\
&\quad +v_1^{-1+\s m^{-1}}\bt v_2^{-\r m^{-1}+\s\r^2 m \r^2}+v_1^{m\r^2}\bt v_2^{m\r^2-\s\r^2}\\
&\quad +v_1^{-1+\s m^{-1}-\r m^{-1}+\s\r^2 m \r^2}
\bt v_2^{-m\r^2+\s\r^2})=
\end{align*}
[by $w^{-1+\s m^{-1}-\r m^{-1}+\s\r^2 m \r^2 -m\r^2+\s\r^2}=0$ for $w=v_1,v_2,u$ (as in proof of Lemma \ref{tr_kr})]
\begin{align*}
&=(1,0,0, v_1\bt v_2^{1-\s m^{-1}+\r m^{-1}-\s\r^2 m \r^2}
+v_1^{\r m^{-1}}\bt v_2^{\r m^{-1}-\s\r^2 m \r^2}\\
&\quad +v_1^{\s m^{-1}}\bt v_2^{-\r m^{-1}+\s\r^2 m \r^2}
+v_1^{m\r^2}\bt v_2^{m\r^2 -\s\r^2}
+v_1^{m \r^2-\s \r^2}\bt v_2^{-m\r^2+\s\r^2})\\
&=(1,0,0, v_1\bt v_2^{-m \r^2+\s\r^2}
+v_1^{\r m^{-1}-\s m^{-1}}\bt v_2^{\r m^{-1}-\s\r^2 m \r^2}
+v_1^{-\s \r^2}\bt v_2^{-m\r^2+\s\r^2})\\
&=(1,0,0, -v_1^{\r^2(1-\s)\r}\bt v_2^{\r m^{-1}\r^2(1-\s)\r^2m^{-1}}
+v_1^{\r^2(1-\s)\r^2 m^{-1}}\bt v_2^{\r m^{-1}\r^2(1-\s)\r}
).
\end{align*}

The elements $l_1=v_1^{\r^2(1-\s)}$ and $l_2=v_2^{\r m^{-1}\r^2(1-\s)}$ run through $\M(V_1)$ and $\M(V_2)$ as
$v_1$ and $v_2$ run through $V_1$ and $V_2$, respectively. Hence, $G$ is a group with triality if and only if
the element
$$-l_1^\r\bt l_2^{\r^2m^{-1}}+l_1^{\r^2 m^{-1}}\bt l_2^\r=(-l+l^\s)^{\r m\r^2}$$
is zero, where $l=l_1^{\r^2m^{-1}\r^2}\bt l_2^{(\r^2m^{-1}\r^2)^2}$. The claim follows.

\end{proof}

\section{The group with triality}

Let $R$ be a commutative unital ring and let $R_0$ be a subgroup of $R^\times$. We set $T=R_0\times R_0$ and
let $S=\la\s,\r\ra$ act on $T$ according to

$$
[\s]=\left(\begin{array}{rr}
                    -1&.\\
                    1&1
                   \end{array}\right), \quad
[\r]=\left(\begin{array}{rr}
                    .&1\\
                    -1&-1
                   \end{array}\right),
$$
e.\,g., $(r_1,r_2)^\s=(r_1^{-1}r_2,r_2)$, $r_i\in R_0$. Then $T$ is a group with triality $S$ with $\M(T)=\{
(r,1)\mid r\in R_0 \}\cong R_0$. Denote $\wt{T}=S\rd T$.

 We define an $R\wt{T}$-module $V$ as a free $R$-module
of rank $3$ with basis $\bm{e}=\{e_1,e_2,e_3\}$ in which the corresponding $R$-representation $\Psi$ for
$\wt{T}$ is given by
\begin{equation}\label{rep}
\begin{array}{rl}
\Psi:&(r_1,r_2)\mapsto\left(\begin{array}{ccc}
                           r_1&.&.\\
                           .&r_1^{-1}r_2&.\\
                           .&.&r_2^{-1}
                       \end{array}\right),\\[20pt]
&\s\mapsto\left(\begin{array}{ccc}
                           .&1&.\\
                           1&.&.\\
                           .&.&1
                       \end{array}\right),\quad
\r\mapsto\left(\begin{array}{ccc}
                           .&1&.\\
                           .&.&1\\
                           1&.&.
                       \end{array}\right).
\end{array}
\end{equation}
Swapping  $r\leftrightarrow r^{-1}$ for $r\in R_0$ gives the matrices of the contragredient representation
$\Psi^*$ in the dual basis  $\bm{e^*}=\{e_1^*,e_2^*,e_3^*\}$ of the corresponding module $V^*$.

\begin{lem} \label{bm}
The $R\wt{T}$-modules $V$ and $V^*$ are triality modules for~$T$.
\end{lem}
\begin{proof} By duality, we may consider $V$ only. A direct verification shows that $\Psi((\s-1)(1+\t+\t^2))$ is the zero
matrix, where $\t$ has the form
\be\label{rrmrr}
\r^2m\r^2=\left(\begin{array}{ccc}
                            .&1&.\\
                           .&.&r\\
                           r^{-1}&.&.
                       \end{array}\right)
\ee
for $m=(r,1)\in \M(T)$, $r\in R_0$. The claim now follows by Lemma \ref{tr_kr}.
\end{proof}

The contragredient module $V^*$ can be realized as a direct summand of the symmetric square $S^2V$ due to the following

\begin{lem} \label{mon} The $2$-homogeneous component of $R[x_1,x_2,x_3]$ splits under the
action of $\wt{T}$ given by $t:x_i\mapsto \sum_j(\Psi(t))_{ij}x_j$, $t\in \wt{T}$ into the direct sum
\be\label{ds} \la x_2x_3,x_1x_3,x_1x_2\ra_R\oplus \la x_1^2,x_2^2,x_3^2\ra_R.\ee The first summand
is isomorphic to $V^*$ as an $R\wt{T}$-module under the map $\g:\ x_ix_j\mapsto e_k^*$ whenever
$\{i,j,k\}=\{1,2,3\}$. The second summand is isomorphic to $V^*$ under $\delta:\ x_i^2\mapsto
e_i^*$ for $i=1,2,3$ provided that $R_0$ has exponent $3$.
\end{lem}
\begin{proof} For every $t\in \wt{T}$, the matrix $\Psi(t)$ is monomial, hence has the form
$\sum_k r_ke_{k,k^{\t}}$ for suitable $r_k\in R$ and $\t\in \rm{Sym}_3$, where $e_{i,j}$ are the matrix units.
Therefore, $t$ acts by $t:x_i\mapsto r_ix_{i^\t}$, which implies the decomposition (\ref{ds}). It also implies
that
$$t:x_ix_j\mapsto r_i r_j x_{i^\t,j^\t}\stackrel{\gamma}{\mapsto}r_k^{-1} e_{k^\t}^*$$
whenever $\{i,j,k\}=\{1,2,3\}$, because for all matrices in (\ref{rep}) we have $r_1r_2r_3=1$; and

$$t:x_i^2\mapsto r_i^2 x_{i^\t}^2\stackrel{\delta}{\mapsto}r_i^{-1} e_{i^\t}^*,$$
if $r_i^3=1$. However, $\Psi^*(t)=\sum_k r_k^{-1}e_{k,k^{\t}}$. Hence, $\g$ extends to an isomorphism onto
$V^*$ and so does $\delta$ whenever $R_0$ has exponent $3$.
\end{proof}

%We also observe that $\Psi^*$ is equivalent to the representation of $\wt{T}$ on the
%of symmetric $3\times 3$-matrices over $R$ with zero diagonal defined by $A\mapsto \Psi(t)^\bot A\Psi(t)$ for all $t\in \wt{T}$.

Let $a,b\in R$ be fixed elements at least one of which is invertible.
We henceforth assume that one of the following conditions is fulfilled:
\begin{enumerate}
\renewcommand{\theenumi}{\Roman{enumi}}
\item $R_0^3=1$,
\item $b=0$.
\end{enumerate}
Then by Lemma \ref{bm} the submodule of $S^2V$ spanned by  $ax_i^2+bx_jx_k$, where $\{i,j,k\}=\{1,2,3\}$
is isomorphic to $V^*$. The invertibility of either of $a,b$ ensures that this submodule is complemented in $S^2V$.
Hence, there is an $R\wt T$-module homomorphism $\vf_{a,b}:V\otimes V\to V^*$ which is written in
the bases $\bm{e}$ and $\bm{e}^*$ as
\be\label{oper2}
\begin{array}{r@{}l}
(v_1,v_2,v_3)\bt(w_1,w_2,w_3)=\big(&a(v_2w_3+v_3w_2)+bv_1w_1,\\[5pt]
a(v_1&w_3+v_3w_1)+bv_2w_2,a(v_1w_2+v_2w_1)+bv_3w_3\big),
\end{array}
\ee
where $\bt=\bt_{a,b}=\vf_{a,b}\circ\otimes$. By the discussion in Section \ref{sect}, we may construct the group $G=T\rd W$,
where $W=V\times V\times V^*$ has
the operation (\ref{oper2}). We will henceforth denote $\wt{G}=\wt{T}\rd W=S\rd G$.

\begin{lem} \label{gtr} We have
\begin{enumerate}
\item[$(i)$] $G$ is a group with triality $S$,
\item[$(ii)$] ${\rm Z}_S(G)=1$ and $[G,S]=G$.
\end{enumerate}
\end{lem}
\begin{proof} $(i)$ Let $l_1,l_2$ be arbitrary elements of $\M(V)$. Then there exist $s_1,s_2\in R$ such
that $l_i=(s_i,-s_i,0)$. Let $m\in \M(T)$. Then $m=(r,1)$ for some $r\in R_0$ and $\r^2m\r^2$ is
as in (\ref{rrmrr}). By (\ref{oper2}), we have
$$
l_1^{\r^2m\r^2}\bt l_2^{(\r^2 m \r^2)^2}=(0,s_1,-s_1r)\bt(-s_2,0,rs_2)=s_1s_2(ar,ar,-a-br^2),
$$
which lies in $C_{V^*}(\s)$. The claim follows by Lemma \ref{trnil2}.

$(ii)$ Clearly, every proper nontrivial normal $S$-invariant subgroup of $G$ must include $V^*$ and be included in $W$.
Since $S$ induces a nontrivial action on both $V^*$ and $G/W$, the claim follows.

\end{proof}

\section{The Moufang loop}

Lemma \ref{gtr} implies the existence of a Moufang loop $\M(G)$ which depends on the parameters $R,R_0,a,b$.
Assuming that $R$ and $R_0$ are fixed, we will denote this loop by $M_{a,b}$ and determine its structure.

\begin{lem} \label{me} The Moufang loop $M_{a,b}$ consists of the elements of $G$ of the form
\be\label{form}
\big((r,1),x(1,-r^{-1},0),y(1,-r^{-1},0),z(-r^{-1},1,0)+xy(b,0,-ar^{-1})\big).
\ee
for $r\in R_0,\ x,y,z\in R$.
\end{lem}
\begin{proof} Let $g=(t,v_1,v_2,u)\in G$. Then $g^{-1}g^\s$ is given by (\ref{gmgs}), where
$m=t^{-1}t^\s=(1,r)\in \M(T)$, $v_i=(v_{i1},v_{i2},v_{i3})$, $i=1,2$, $u=(u_1,u_2,u_3)$. Due to
$$\Psi(-m+\s)=\left(\begin{array}{ccc}
                            -r&1&.\\
                            1&-r^{-1}&.\\
                            .&.&.
                       \end{array}\right),$$
we have
\begin{align*}
v_1^{-m+\s}&=v_1\Psi(-m+\s)=x(1,-r^{-1},0),\quad \text{where}\ x=-rv_{11}+v_{12},\\
v_2^{-m+\s}&=v_2\Psi(-m+\s)=y(1,-r^{-1},1,0),\quad \text{where}\ y=-rv_{21}+v_{22},\\
u^{-m+\s}&=v_i\Psi^*(-m+\s)=w(-r^{-1},1,0),\quad \text{where}\ w=u_1-ru_2,\\
v_1^m\bt v_2^{m-\s}&=(rv_{11},r^{-1}v_{12},v_{13})\bt(-y,r^{-1}y,0)=(z_1,z_2,-ar^{-1}xy), \quad \text{where}\\
z_1&=y(ar^{-1}v_{13}-brv_{11}), \ z_2=y(-av_{13}+br^{-2}v_{12}).
\end{align*}
Since we assume (I) or (II), we have $br^3=b$ and so
$$
z_1+r^{-1}z_2=y(-brv_{11}+br^{-3}v_{12})=bxy.
$$
Hence, setting $z=w+z_2$, we have the required form of $g^{-1}g^\s$.
\end{proof}
By Lemma \ref{me}, sending an element (\ref{form}) to the tuple
$$
(r,x,y,z)\quad r\in R_0,\ x,y,z\in R
$$
gives a bijection from $M_{a,b}$. We will therefore assume that $M_{a,b}$ consists of all such tuples.

\begin{lem} \label{mult} The multiplication and inversion in the Moufang loop $M_{a,b}$ are given by
\begin{align*}
(r_1,x_1,&y_1,z_1)(r_2,x_2,y_2,z_2)=\\
(r_1&r_2,x_1+r_1x_2,y_1+r_1y_2,r_2z_1+z_2+a(x_1y_2-x_2y_1)+br_1^{-1}r_2x_1y_2),\\
(r,x,&y,z)^{-1}=(r^{-1},-r^{-1}x,-r^{-1}y,-r^{-1}z+bxy).
\end{align*}
\end{lem}
\begin{proof} We first derive the inversion formula using (\ref{pi}). Denoting $m=(r,1)$ we have
\begin{align*}
&(r,x,y,z)^{-1}\leftrightarrow\big(m,x(1,-r^{-1},0),y(1,-r^{-1},0),z(-r^{-1},1,0)+xy(b,0,-ar^{-1})\big)^{-1}\\
&=\big(m^{-1},-x(1,-r^{-1},0)^{m^{-1}},-y(1,-r^{-1},0)^{m^{-1}},-z(-r^{-1},1,0)^{m^{-1}}\\
&\quad -xy(b,0,-ar^{-1})^{m^{-1}}+ x(1,-r^{-1},0)^{m^{-1}}\bt y(1,-r^{-1},0)^{m^{-1}} \big)\\
&=\big(m^{-1},-x(r^{-1},-1,0),-y(r^{-1},-1,0),-z(-1,r^{-1},0)-xy(br,0,-ar^{-1})\\
&\quad+xy(br^{-2},b,-2ar^{-1})\big)=\big((r^{-1},1),-r^{-1}x(1,-r,0),-r^{-1}y(1,-r,0),\\
&\quad-r^{-1}z(-r,1,0)+bxy(-r,1,0)+(-r^{-1}x)(-r^{-1}y)(b,0,-ar)\big)\\
&\leftrightarrow(r^{-1},-r^{-1}x,-r^{-1}y,-r^{-1}z+bxy).
\end{align*}
Similarly, with $m_i=(r_i,1)$, $i=1,2$, we have
\begin{align*}
&(r_1,x_1,y_1,z_1)(r_2,x_2,y_2,z_2)\leftrightarrow \big(m_1,x_1(1,-r_1^{-1},0),y_1(1,-r_1^{-1},0),z_1(-r_1^{-1},1,0)\\
&\quad+x_1y_1(b,0,-ar_1^{-1})\big)^{-\r}\big(m_2,x_2(1,-r_2^{-1},0),y_2(1,-r_2^{-1},0),z_2(-r_2^{-1},1,0)\\
&\quad+x_2y_2(b,0,-ar_2^{-1})\big)\big(m_1,x_1(1,-r_1^{-1},0),y_1(1,-r_1^{-1},0),z_1(-r_1^{-1},1,0)\\
&\quad+x_1y_1(b,0,-ar_1^{-1})\big)^{-\r^2}=\big(m_1^{-\r},x_1(-r_1^{-1},1,0)^\r,y_1(-r_1^{-1},1,0)^\r,z_1(1,-r_1^{-1},0)^\r\\
&\quad+x_1y_1(0,b,-ar_1^{-1})^\r\big)\big(m_2,x_2(1,-r_2^{-1},0),y_2(1,-r_2^{-1},0),z_2(-r_2^{-1},1,0)\\
&\quad+x_2y_2(0,b,-ar_2)\big)\big(m_1^{-\r^2},x_1(1,-r_1,0)^{\r^2},y_1(1,-r_1,0)^{\r^2},z_1(-r_1,1,0)^{\r^2}\\
&\quad+x_1y_1(0,b,-ar_1^{-1})^{\r^2}\big)=\big((r_2,r_1^{-1}),x_1(0,-r_1^{-1},1)^{m_2}+x_2(1,-r_2^{-1},0),\\
&\quad y_1(0,-r_1^{-1},1)^{m_2}+y_2(1,-r_2^{-1},0),z_1(0,1,-r_1^{-1})^{m_2}+z_2(-r_2^{-1},1,0)\\
&\quad +x_1y_1(-ar_1^{-1},0,b)^{m_2}+x_2y_2(b,0,-ar_2^{-1})+x_1y_2(0,-r_1^{-1},1)^{m_2}\bt(1,-r_2^{-1},0)\big)\\
&\times\big((r_1,r_1),x_1(1,0,-r_1^{-1}),y_1(1,0,-r_1^{-1}),z_1(-r_1^{-1},0,1)+x_1y_1(b,-ar_1^{-1},0)\big)\displaybreak\\
&=\big((r_1r_2,1),(x_2,-r_1^{-1}r_2^{-1}x_1-r_2^{-1}x_2,x_1)^{m_1^{-\r^2}}+x_1(1,0,-r_1^{-1}),\\
&\quad (y_2,-r_1^{-1}r_2^{-1}y_1-r_2^{-1}y_2,y_1)^{m_1^{-\r^2}}+y_1(1,0,-r_1^{-1}),
(-r_2^{-1}z_2,r_2z_1+z_2,-r_1^{-1}z_1)^{m_1^{-\r^2}}\\
&\quad +z_1(-r_1^{-1},0,1)+(-ar_1^{-1}r_2^{-1}x_1y_1+bx_2y_2,0,bx_1y_1-ar_2^{-1}x_2y_2)^{m_1^{-\r^2}}\\
&\quad +x_1y_1(b,-ar_1^{-1},0)+x_1y_2(-ar_2^{-1},a+br_1^{-1}r_2,-ar_1^{-1}r_2^{-1})^{m_1^{-\r^2}}\\
&\quad +(x_2,-r_1^{-1}r_2^{-1}x_1-r_2^{-1}x_2,x_1)^{m_1^{-\r^2}}\bt y_1(1,0,-r_1^{-1})\big)\\
&=\big((r_1r_2,1),(x_1+r_1x_2)(1,-r_1^{-1}r_2^{-1},0),(y_1+r_1y_2)(1,-r_1^{-1}r_2^{-1},0),\\
&\quad  (r_2z_1+z_2)(-r_1^{-1}r_2^{-1},1,0)+(-ar_1^{-2}r_2^{-1}x_1y_1+br_1^{-1}x_2y_2,0,br_1x_1y_1-ar_1r_2^{-1}x_2y_2)\\
&\quad +x_1y_1(b,-ar_1^{-1},0)+x_1y_2(-ar_1^{-1}r_2^{-1},a+br_1^{-1}r_2,-ar_2^{-1})\\
&\quad +(a(r_1^{-2}r_2^{-1}x_1y_1+r_1^{-1}r_2^{-1}x_2y_1)+br_1x_2y_1,a(-x_2y_1+r_1^{-1}x_1y_1),\\
&\quad a(-r_1^{-1}r_2^{-1}x_1y_1-r_2^{-1}x_2y_1)-br_1^{-2}x_1y_1)\big)=\big((r_1r_2,1),(x_1+r_1x_2)(1,-r_1^{-1}r_2^{-1},0),\\
&(y_1+r_1y_2)(1,-r_1^{-1}r_2^{-1},0),(r_2z_1+z_2+a(x_1y_2-x_2y_1)+br_1^{-1}r_2x_1y_2)(-r_1^{-1}r_2^{-1},1,0)\\
&\quad  +(x_1+r_1x_2)(y_1+r_1y_2)(b,0,-ar_1^{-1}r_2^{-1})\big)\leftrightarrow(r_1r_2,x_1+r_1x_2,y_1+r_1y_2,\\
&\quad r_2z_1+z_2+a(x_1y_2-x_2y_1)+br_1^{-1}r_2x_1y_2).
\end{align*}
\end{proof}

By Lemma \ref{mult}, the loop $M_{a,b}$ has the form $R_0.N$,
where $N=(R+R).R$ is a normal subgroup. It can be checked that
the subgroup $N$ is commutative if and only if $2a+b=0$. The associator of $(r_i,x_i,y_i,z_i)\in M_{a,b}$, $i=1,2,3$,
is trivial if and only if
$$
a\,\left|\begin{array}{ccc}
                           r_1-1&r_2-1&r_3-1\\
                           x_1&x_2&x_3\\
                           y_1&y_2&y_3
                       \end{array}\right|=0.
$$
In particular, $M_{a,b}$ is nonassociative, if $a\ne 0$, $|R_0|>1$, and $R$ is a domain, in
which case ${\rm Nuc}(M_{a,b})$ consists of the elements $(1,0,0,z)$, $z\in R$.

We observe that, for every $c\in R^\times$, the loops $M_{a,b}$ and $M_{ca,cb}$ are isomorphic, which can be shown
by changing the "coordinates" $(r,x,y,z)\mapsto(r,x,y,cz)$ in $M_{a,b}$. Hence, we may only consider the loops
$M_{1,b}$ and $M_{a,1}$. In particular, we obtain
abelian-by-cyclic Moufang loops, $M_{1,-2}$, with the multiplication (\ref{abm})
which split into two types:
\begin{enumerate}
\renewcommand{\theenumi}{\Roman{enumi}}
\item if the characteristic of $R$ is not $2$ then
$M_{1,-2}$ has the form $3.(R+R+R)$, i.e. $R_0$ is cyclic of order $3$,
\item if the characteristic of $R$ is  $2$ then
$M_{1,-2}=M_{1,0}$ has the form $R_0.N$, where  $R_0$ is an arbitrary cyclic subgroup of $R^\times$ and
$N=R+R+R$ is an elementary abelian $2$-group.
\end{enumerate}

\section{Embedding in the Cayley algebra}

We show that a particular case of the above-constructed series of Moufang loops, namely $M_{1,0}$,
can be embedded in a Cayley algebra. Recall that the Cayley algebra $\mathbb{O}=\mathbb{O}(R)$ can be defined
as set of all {\em Zorn matrices}
$$
\left( \begin{array}{cc}
  a & {\bf v} \\
 {\bf w} &  b \\
\end{array}
\right), \ \  a, b \in R, \ \ {\bf v},{\bf w}\in R^3
$$
with the natural structure of a free $R$-module and
multiplication given by the rule
\begin{equation} \label{cayley_mult}
\begin{array}{r@{}l}
\left( \begin{array}{cc}
  a_1 & {\bf v}_1 \\
 {\bf w}_1 &  b_1 \\
\end{array}
\right)\cdot
\left( \begin{array}{cc}
 a_2 & {\bf v}_2 \\
 {\bf w}_2 & b_2 \\
\end{array}
\right) =& \\[20pt]
\left( \begin{array}{cc}
 a_1 a_2+ {\bf v}_1 \cdot {\bf w}_2 &  a_1 {\bf v}_2 + b_2 {\bf v}_1 \\
 a_2 {\bf w}_1 +  b_1 {\bf w}_2 &  {\bf w}_1\cdot {\bf v}_2 +  b_1 b_2 \\
\end{array}
\right)& +
\left( \begin{array}{cc}
 0 & - {\bf w}_1\times {\bf w}_2 \\
 {\bf v}_1\times {\bf v}_2 & 0 \\
\end{array}
\right),
\end{array}
\end{equation}
where, for ${\bf v}=(v_1,v_2,v_3)$ and ${\bf w}=(w_1,w_2,w_3)$ in
$R^3$, we denoted
$$
\begin{array}{c}
{\bf v}\cdot{\bf w}= v_1w_1+v_2w_2+v_3w_3\in R,\\
{\bf v}\times{\bf w}=(v_2w_3-v_3w_2,v_3w_1-v_1w_3,v_1w_2-v_2w_1)\in R^3.
\end{array}
$$
It is well-known that $\mathbb{O}$ is an alternative algebra and the set of its invertible
elements $\mathbb{O}^\times$ forms a Moufang loop.

Consider the subset of $\mathbb{O}^\times$ of elements of the form
$$
\left( \begin{array}{cc}
  r & (0,x,y) \\
 (z,0,0) &  1 \\
\end{array}
\right)
$$
which we identify with the tuples $(r,x,y,z)$, where $r\in R_0\le R^\times$, $x,y,z\in R$. Using
(\ref{cayley_mult}) it can be checked that this subset is a subloop with the multiplication
\begin{align*}
(r_1,x_1,& y_1,z_1)(r_2,x_2,y_2,z_2)=\\
(r_1&r_2,x_1+r_1x_2,y_1+r_1y_2,r_2z_1+z_2+x_1y_2-x_2y_1),\\
\end{align*}
hence is isomorphic to $M_{1,0}$.

\section{The isomorphy problem}

Clearly, the loops $M_{a,b}$ and $M_{a',b'}$ are isomorphic if $(a',b')=(a,b)^\varphi$ for some
$\varphi\in {\rm Aut}(R)$. However, determining all isomorphisms among the loops $M_{a,b}$ seems to be
a challenging problem. In particular, we state

\begin{prob} \label{pris} Can the loops $M_{1,b}$ and $M_{1,b'}$ be isomorphic for non-${\rm Aut}(R)$-conjugate
elements $b,b'\in R$?
\end{prob}

We only prove the following particular result.

\begin{prop} \label{nis} Let $R$ be a field and let $0\ne b\in R$. Then $M_{1,0}$ is not isomorphic to $M_{1,b}$.
\end{prop}
\begin{proof} By lemma \ref{gtr}(ii), the above-constructed group with triality $G=G_{a,b}$ coincides with $\E(M_{a,b})$.
By Lemma \ref{mmini}, it suffices to show that $G_{1,0}$ is not isomorphic to $G_{1,b}$.
Observe that, since $R_0^3=1$
% this is because $M_{1,b}$ is defined for $b\ne 0$,
and  $R$ is a field, we have $|R_0|=3$. By the construction of $G$, we have $W=[G,G]$, $V^*={\rm Z}(W)$,
and $W/V^*\cong V\oplus V$ is the direct sum of three $RT$-homogeneous $2$-dimensional components $Q_i$, $i=1,2,3$,
spanned by the pairs $(e_i,0)$ and $(0,e_i)$,  where $e_i$ is the $i$th basis vector for $V$. By (\ref{pi}) and (\ref{oper2}),
the full preimage $\widehat{Q}_i$ of $Q_i$ in $W$ can be identified with a group with the multiplication
$$
(r_1,r_2,u)(r_1',r_2',u')=(r_1+r_1',r_2+r_2',u+u'+br_1r_2'e_i^*),
$$
where $r_j,r_j'\in R$ and $u,u'\in V^*$. Hence, the groups $\widehat{Q}_i$, $i=1,2,3$, are abelian if
and only if $b=0$. Due to the invariant way these groups were constructed, we conclude that
$G_{1,0}$ and $G_{1,b}$ are not isomorphic.
\end{proof}


\begin{thebibliography}{100}

\bibitem{bru}  R. H. Bruck, A survey of binary systems. Springer-Verlag. 1958.

\bibitem{che} O. Chein, Moufang loops of small order, I, {\it Trans. Am. Math. Soc.} {\bf
188}, (1974), 31--51.

\bibitem{chr} O. Chein, A. Rajah, Possible orders of nonassociative Moufang loops,
{\em Comment. Math. Univ. Carolin.}, {\bf 41},  N\,2 (2000), 237–244.

\bibitem{dor}  S. Doro, Simple Moufang loops,
{\em Math. Proc. Camb. Phil. Soc.},
{\bf 83}, (1978), 377-392.

\bibitem{gag} S. M. Gagola III, Hall's theorem for Moufang loops.
{\em J. Algebra}, {\bf 323}, N\,12 (2010), 3252-–3262.

\bibitem{gz} A. N. Grishkov, A. V. Zavarnitsine, Lagrange's
theorem for Moufang loops, {\em Math. Proc. Camb. Phil. Soc.}, {\bf 139}, N\,1 (2005),
41--57.

\bibitem{gz_tri}  A. N. Grishkov, A. V. Zavarnitsine, Groups with triality, {\em J. Algebra Appl.},
{\bf 5}, N\,4 (2006), 441--463.

\bibitem{syl} A. N. Grishkov, A. V. Zavarnitsine, Sylow's theorem for Moufang loops, J. Algebra,
{\bf 321}, N\,7 (2009), 1813--1825.

\bibitem{lie}  M. W. Liebeck, The classification of finite simple
Moufang loops, {\em Math. Proc. Camb. Phil. Soc.}, {\bf 102},
(1987), 33--47.

\bibitem{raj} A. Rajah, Moufang loops of odd order $pq^3$, J. Algebra, {\bf 235}, N\,1 (2001), 66-–93.



\end{thebibliography}
\end{document}